\documentclass[reqno,11pt]{amsart}

\usepackage[cmtip,arrow]{xy}  
\usepackage{pb-diagram,pb-xy}  
\usepackage{graphicx}    
\usepackage{amsfonts}
\usepackage{amsmath, amssymb}
\usepackage[latin1]{inputenc}
\usepackage{color}

\usepackage{graphicx}
\usepackage{url}
\usepackage{tikz}
\usetikzlibrary{patterns}

\usepackage{color}
\usepackage{pgfplots}
\usepackage{subfigure}
\usepackage{caption}
\usepackage{float}

\usepackage{amsmath}
\usepackage{amsthm}
\usepackage{amsfonts}
\usepackage{array}
\newcolumntype{C}[1]{>{\centering\arraybackslash}p{#1}}
\usetikzlibrary{matrix, calc, decorations.pathmorphing, decorations.markings, shapes.arrows, shapes, snakes}

\newcommand{\suchthat}{\;\ifnum\currentgrouptype=16 \middle\fi|\;}

\newcommand\restr[2]{{
  \left.\kern-\nulldelimiterspace
  #1
  \vphantom{\big|}
  \right|_{#2}
  }}

\newcommand\DrawGenus[7]{
  \pgfmathsetmacro{\xstart}{#1 - (0.985*#4)}
  \pgfmathsetmacro{\ystart}{#2 + (0.2*#3)}
	\draw[color = #6, rotate around={#5:(#1,#2)}, #7] (\xstart, \ystart) arc (190:350:#4  and #3);
	\draw[color = #6, rotate around={#5:(#1,#2)}, #7] (\xstart, \ystart) arc (190:210:#4  and #3) arc (150:30:#4  and #3) arc (330:350:#4  and #3);
}

\newcommand\DrawFilledGenus[8]{
  \pgfmathsetmacro{\xstart}{#1 - (0.985*#4)}
  \pgfmathsetmacro{\ystart}{#2 + (0.2*#3)}
	\draw[color = #6, rotate around={#5:(#1,#2)}, #7] (\xstart, \ystart) arc (190:350:#4  and #3);
	\draw[color = #6, rotate around={#5:(#1,#2)}, #7] (\xstart, \ystart) arc (190:210:#4  and #3) arc (150:30:#4  and #3) arc (330:350:#4  and #3);
	\draw[color = #6, rotate around={#5:(#1,#2)}, #7, fill = #8] (\xstart, \ystart) arc (190:210:#4  and #3) arc (150:30:#4  and #3) arc (-30:-150:#4  and #3);
}

\newcommand\DrawDonut[7]{
  \pgfmathsetmacro{\fctr}{.08}
  \pgfmathsetmacro{\newwidth}{0.5*#4}
  \pgfmathsetmacro{\newheight}{0.5*#3}
  \draw[color = #6, rotate around={#5:(#1,#2)}, #7] (#1, #2) ellipse (#4  and #3);
  \DrawGenus{#1}{#2}{\newheight}{\newwidth}{#5}{#6}{#7}
}

\newcommand\DrawFilledDonutops[8]{
  \pgfmathsetmacro{\fctr}{.08}
  \pgfmathsetmacro{\newwidth}{0.5*#4}
  \pgfmathsetmacro{\newheight}{0.5*#3}
  \draw[color = #6, rotate around={#5:(#1,#2)}, #7, fill = #6, opacity = .6] (#1, #2) ellipse (#4  and #3);
  \DrawFilledGenus{#1}{#2}{\newheight}{\newwidth}{#5}{#6}{#7}{#8}
}

\tikzstyle{mytheorembox} = [draw=vdgreen, fill=blue!20, very thick, rectangle, rounded corners, inner sep=10pt, inner ysep=10pt]
\tikzstyle{mytheoremfancytitle} =[fill=vdgreen, text=white]

\definecolor{vdblue}{rgb}{0,0,.3}
\definecolor{dblue}{rgb}{0,0,.7}
\definecolor{lblue}{rgb}{.3,.3,1}
\definecolor{vlblue}{rgb}{.7,.7,1}
\definecolor{vvlblue}{rgb}{.9,.9,1}

\definecolor{vdred}{rgb}{.3,0,0}
\definecolor{dred}{rgb}{.7,0,0}
\definecolor{lred}{rgb}{1,.3,.3}
\definecolor{vlred}{rgb}{1,.7,.7}

\definecolor{vdgreen}{rgb}{0,.2,0}
\definecolor{dgreen}{rgb}{0,.4,0}
\definecolor{lgreen}{rgb}{.3,1,.3}
\definecolor{vlgreen}{rgb}{.7,1,.7}

\definecolor{lyellow}{rgb}{1,1,.3}

\definecolor{gray1}{rgb}{0.22,0.22,0.22}
\definecolor{gray2}{rgb}{0.28,0.28,0.28}
\definecolor{gray3}{rgb}{0.36,0.36,0.36}
\definecolor{gray4}{rgb}{0.44,0.44,0.44}
\definecolor{gray5}{rgb}{0.52,0.52,0.52}
\definecolor{gray6}{rgb}{0.6,0.6,0.6}
\definecolor{gray7}{rgb}{0.68,0.68,0.68}
\definecolor{gray8}{rgb}{0.76,0.76,0.76}

\definecolor{color1}{rgb}{1,0,0}
\definecolor{color2}{rgb}{0.98,0,0.816}
\definecolor{color3}{rgb}{0.717,0,1}
\definecolor{color4}{rgb}{0,0,1}
\definecolor{color5}{rgb}{0,1,1}
\definecolor{color6}{rgb}{0,1,0}
\definecolor{color8}{rgb}{1,1,0}
\definecolor{color7}{rgb}{1,0.651,0}

\newcommand{\SF}{\varepsilon}

\newcommand{\T}{\mathbb{T}}

\renewcommand{\SF}{{\mathcal{F}}}

\newcommand{\SL}{{\mathcal{L}}}

\newcommand{\SO}{{\mathcal{O}}}

\newcommand{\ST}{{\mathcal{T}}}

\newcommand{\Z}{\mathbb{Z}}

\newcommand{\R}{\mathbb{R}}

\renewcommand{\S}{\mathbb{S}}

\newcommand{\id}{\textup{id}}

\newcommand{\const}{\textup{const}}

\renewcommand{\mod}{\textup{mod}}

\newcommand{\without}[1]{\backslash\{#1\}}

\newcommand{\bd}{\partial}

\newcommand{\frbd}[1]{\frac{\bd}{\bd #1}}
\newcommand{\frbdt}[2]{\frac{\bd #1}{\bd #2}}

\newcommand{\norm}[1]{\|#1\|}

\newtheorem{proposition}{Proposition}
\newtheorem{theorem}[proposition]{Theorem}
\newtheorem{definition}[proposition]{Definition}

\newtheorem{example}[proposition]{Example}

\theoremstyle{remark}
\newtheorem{remark}[proposition]{Remark}

\begin{document}

\bibliographystyle{alpha}

\title[ Action-angle variables and  KAM  for $b$-Poisson manifolds]{Action-angle variables and a KAM  theorem for $b$-Poisson manifolds}

\author{Anna Kiesenhofer}
\address{Departament de Matem\`{a}tica Aplicada I, Universitat Polit\`ecnica de Catalunya, EPSEB, Avinguda del Doctor Mara\~{n}\'{o}n 44--50, Barcelona, Spain}
\email{anna.kiesenhofer@upc.edu}
\author{Eva Miranda}
\address{Departament de Matem\`{a}tica Aplicada I,
 Universitat Polit\`{e}cnica de Catalunya, EPSEB, Avinguda del Doctor Mara\~{n}\'{o}n 44--50, Barcelona, Spain}
\email{eva.miranda@upc.edu}
\thanks{Anna Kiesenhofer is supported by AGAUR FI doctoral grant. Anna Kiesenhofer and Eva Miranda are partially supported by Ministerio de Econom\'{i}a y Competitividad project GEOMETRIA ALGEBRAICA, SIMPLECTICA, ARITMETICA Y APLICACIONES with reference: MTM2012-38122-C03-01/FEDER and by the European Science Foundation network CAST. Geoffrey Scott was supported by a special grant Ajut Mobilitat EPSEB 2014.}
\author{Geoffrey Scott}
\address{ Department of Mathematics, University of Toronto, Canada}
\email{gscott@math.utoronto.ca}
\date{\today}
\begin{abstract}
  In this article we prove an action-angle  theorem for $b$-integrable systems on $b$-Poisson manifolds improving the action-angle theorem contained in  \cite{Laurent-Gengoux2010} for general  Poisson manifolds in this setting. As an application, we prove a KAM-type theorem for $b$-Poisson manifolds.
\end{abstract}
\maketitle


\section{Introduction}

The classic Hamilton's equations are the  equations of the flow of  a Hamiltonian vector field $X_H$ determined by a Hamiltonian function $H$ and a Darboux symplectic form in position and momenta $\omega_0=\sum_i dp_i\wedge dq_i$ via the correspondence $\iota_{X_H} \omega_0= -dH$.
Symplectic geometry generalizes these equations to the general scenario of Hamiltonian systems associated to a closed non-degenerate $2$-form $\omega$ (general symplectic form).

Among the class of Hamiltonian systems,  the sub-class of integrable systems plays a central role. An integrable system on a $2n$-dimensional symplectic manifold is given by $n-1$ additional first integrals $f_i$ with the property that each integral (including $H$) is preserved by the Hamiltonian flow of the other integrals. This condition is classically known as \emph{involutivity} of the first integrals and can be written in terms of the Poisson bracket as $\{f_i, f_j\}=0$.

Given an integrable system on a symplectic manifold there exist  privileged coordinates in a neighbourhood of an invariant compact submanifold called action-angle coordinates (\cite{mineur}, \cite{landau} and \cite{arnold}\footnote{ The action integrals already appear in a paper by Einstein \cite{einstein}. We refer to \cite{fejoz} for a historical review of the contributions to the action-angle theorem. }). These coordinates not only lead to a direct integration of the Hamiltonian system \cite{liouville} but also they completely describe the geometry of the set of invariant submanifolds. The invariant manifolds are indeed tori named after Liouville. Action-angle variables  describe the set of invariant submanifolds as a  fibration by tori in a neighbourhood of a prescribed Liouville torus.
 The Liouville-Mineur-Arnold theorem also proves that the symplectic structure in this fibration has a Darboux-type form and that the Hamiltonian motion is linear when restricted to the invariant tori.

A generalization of the Liouville-Mineur-Arnold theorem for general Poisson manifolds is contained in \cite{Laurent-Gengoux2010} where a  set of action-angle coordinates is obtained for regular points.  Poisson manifolds \cite{Poisson} constitute a natural generalization of symplectic manifolds that takes singularities (which may correspond to equilibria) into account. Poisson geometry has constituted a  field of research on its own in the last decades. Many Hamiltonian systems (including integrable systems) are naturally formulated in this Poisson setting (see \cite{Laurent-Gengoux2010} and \cite{miranda}).

In this paper we consider a particular class of Poisson manifolds which have been called in the literature $b$-symplectic manifolds, $b$-Poisson manifolds, and log-symplectic manifolds (see  \cite{NestandTsygan}, \cite{Melrose}, \cite{Guillemin2011}, \cite{Guillemin2012}, \cite{Guillemin2013}, \cite{gualtierili}, \cite{gualtierietal}, \cite{marcut}).
The study of $b$-symplectic manifolds is motivated by the study of deformation quantization and pseudodifferential operators on symplectic  manifolds with boundary (see for instance \cite{NestandTsygan} and \cite{Melrose}). A natural set of examples of integrable systems  on $b$-symplectic manifolds comes from this source.

  In this article we improve the action-angle theorem in \cite{Laurent-Gengoux2010} for $b$-Poisson manifolds by extending the construction in  \cite{Laurent-Gengoux2010} to the singular set determined by the $b$-Poisson structure. Our result is optimal in the sense that the number of action cooordinates is half the dimension of the manifold. Our proof strongly uses the main ideas of Duistermaat in \cite{duistermaat} and \cite{Laurent-Gengoux2010}.

  One of the sources of examples of  Poisson manifolds with dense symplectic leaves comes from celestial mechanics where a \emph{singular} change of coordinates takes a Darboux symplectic form to a symplectic form which \lq\lq blows up\rq\rq on a meagre singular set  (see example 3.1.2 in this paper, other examples are contained in \cite{mcgehee} and \cite{devaney}). This includes $b$-symplectic manifolds but also other more degenerate examples as $b^k$-symplectic manifolds (see \cite{scott}).

  In these examples coming from celestial mechanics it is crucial to understand which geometry and dynamics of the integrable system persists under \emph{small} perturbations of the system. In particular to determine whether there are invariant tori and if quasiperiodic trajectories are preserved. Classical KAM theory studies these phenomena in the symplectic zone of the manifold but does not take care of the behaviour at points on the singular set.  In this paper, we address this problem by proving a KAM theorem for Liouville tori in the singular set of a $b$-Poisson manifold\footnote{KAM theorems in the Poisson context have been an object of investigation by Fortunati and Wiggins (see \cite{wiggins}). Their result applied to the $b$-case only considers regular points.}.

{\bf {Organization of this paper:}} Section 2 is devoted to preliminaries on action-angle coordinates, $b$-Poisson manifolds and KAM theory. Section 3 concentrates on integrable systems on $b$-symplectic manifolds and a catalogue of   examples of Hamiltonian systems in $b$-Poisson manifolds. Section 4 contains the statement and the proof for the action-angle theorem for $b$-Poisson manifolds. Section 5 contains a KAM theorem for $b$-Poisson manifolds.

{\bf {Acknowledgements:}} We are deeply thankful to Amadeu Delshams and Victor Guillemin for their contagious enthusiasm on this topic and for many stimulating discussions which in particular brought our attention to the examples contained in \cite{delshams}. Many thanks to  Pol Vanhaecke for his suggestions on  a first version of this paper. We are indebted to Alain Albouy and Jacques F\'{e}joz for their comments and corrections of  the paper and for directing our attention to \cite{einstein} and \cite{fejoz} concerning contributions to the action-angle theorem. Also many thanks  are due to Carles Sim\'{o} for very interesting  remarks concerning the KAM theorem in this paper.

\section{Preliminaries}\label{sec:pre}


\subsection{Action-angle coordinates for Poisson manifolds}

We start by  recalling the definition of an integrable system on a Poisson manifold.

\begin{definition}[Integrable system on a Poisson manifold]\label{def:intsyspoisson}
 Let $(M,\Pi)$ be a Poisson manifold of (maximal) rank $2r$ and of dimension $N$. An $s$-tuple of functions
  $F=(f_1,\dots,f_s)$ on $M$ is a {\bf (Liouville) integrable system} on $(M,\Pi)$ if
  \begin{enumerate}
    \item $f_1,\dots,f_s$ are independent (i.e., their differentials are independent on a dense open subset of $M$);
    \item $f_1,\dots,f_s$ are in involution (pairwise);
    \item $r+s=N$.
  \end{enumerate}
  Viewed as a map, $F:M\to\R^s$ is called the {\bf momentum map} of $(M,\Pi,F)$.
\end{definition}

Consider a point $m\in M$  where the $df_i$ are independent and the rank of $\Pi$ is $2r$.  By the Frobenius theorem,  the Hamiltonian vector fields $X_{f_1},\dots,X_{f_s}$  which span an involutive distribution define an integrable distribution at $m$. Denote by  $\SF_m$  the integral manifold of this distribution passing through $m$. When the $r$-dimensional manifold $\SF_m$ is compact, the action-angle coordinate theorem  proved in  \cite{Laurent-Gengoux2010} (Theorem 1.1 in the paper) says the following:

\begin{theorem}\label{thm:action-angle_intro} 
  Let $(M,\Pi)$ be a Poisson manifold of dimension $N$ and (maximal) rank~$2r$. Suppose that $F=(f_1,\dots,f_s)$
  is an integrable system on $(M,\Pi)$ and assume $m$ and $\SF_m$ satisfy the conditions above.

  Then there exist ${\R} $-valued smooth functions $(\sigma_1,\dots, \sigma_{s})$ and $ {\R}/{\Z}$-valued smooth
  functions $({\theta_1},\dots,{\theta_r})$ defined in a neighborhood $U$ of $\SF_m$ such that
  \begin{enumerate}
    \item the functions $(\theta_1,\dots,\theta_r,\sigma_1,\dots,\sigma_{s})$ define an isomorphism
      $U\simeq\T^r\times B^{s}$, where $B^s$ is an open ball in $\mathbb{R}^s$;
    \item the Poisson structure can be written in terms of these coordinates as
      $$
        \Pi=\sum_{i=1}^r \frbd{\theta_i}\wedge \frbd{\sigma_i}.
      $$
      In particular, the functions $\sigma_{r+1},\dots,\sigma_{s}$ are Casimirs of $\Pi$ on $U$;
    \item the leaves of the surjective submersion $F=(f_1,\dots,f_{s})$ are the level sets of the projection of $\T^r\times B^{s}$ onto $B^{s}$; in particular, the functions $\sigma_1,\dots,\sigma_{s}$ depend on the functions $f_1,\dots,f_{s}$ only.
  \end{enumerate}

\end{theorem}

\begin{remark}
 Note that in the case where $(M,\Pi)$ is symplectic, i.e. $r = \dim(M)/2 = N/2 $ the result above is the classical Liouville-Arnold-Mineur theorem: For a point $m$ as above, the integral manifold $\SF_m$ is an $r$-torus and on a neighborhood around $\SF_m$ we have a locally trivial fibration by $r$-tori, such that the flow of the integrable system is linear on the tori.
\end{remark}


\subsection{$b$-Poisson manifolds}

A particular class of Poisson manifolds, which are very close to being symplectic, are the so called {\bf $b$-Poisson manifolds}.

\begin{definition}\label{def:bpoisson}
Let $(M^{2n},\Pi)$ be a Poisson manifold. If the map
$$p\in M\mapsto(\Pi(p))^n\in\bigwedge^{2n}(TM)$$
is transverse to the zero section, then $\Pi$ is called a \textbf{$b$-Poisson structure} on $M$. The pair $(M,\Pi)$ is called a \textbf{$b$-Poisson manifold}. The vanishing set of $\Pi^n$ is a hypersurface, which we denote by $Z$ and call the {\bf exceptional hypersurface} of $(M,\Pi)$.
\end{definition}

\begin{example}
 On the  2-sphere $\mathbb S^2$ with the usual coordinates $(h,\theta)$, the Poisson structure $\Pi=h\frac{\partial}{\partial h}\wedge\frac{\partial}{\partial \theta}$ vanishes transversally along the equator $Z=\{h=0\}$ and thus it defines a $b$-Poisson structure on $(\mathbb S^2,Z)$.

\end{example}

\begin{example}\label{example2}
For higher dimensions we may consider the following product structures: let $(\mathbb S^2,Z)$ be the sphere in the example above and $(S^{2n},\pi_S)$ be a symplectic manifold, then $(\mathbb S^2\times S,\pi_{\mathbb S^2}+\pi_S)$ is a $b$-Poisson manifold of dimension  $2n+2$.  We may replace $(\mathbb S^2,Z)$ by a $(R,\pi_R)$   Radko compact surface (see \cite{Guillemin2012}).

\end{example}

If $t$ is a local defining function\footnote{There always exists a global defining function if $M$ is oriented, since we can divide $\bigwedge^n \Pi$ by a global nowhere-vanishing section of $\bigwedge^{2n} TM$.} of $Z$, then locally around a point of $Z$ a $b$-Poisson structure has the form
\[
\Pi= \sum_{i=1}^{n-1} \frac{\partial}{\partial x_i} \wedge \frac{\partial}{\partial y_i} + t \frac{\partial}{\partial t} \wedge \frac{\partial}{\partial z}.
\]

\vspace{0.5cm}

One aim of this paper is to show that for $b$-Poisson manifolds  $(M^{2n},\Pi)$ for which the exceptional hypersurface has a trivial normal bundle, we can obtain a stricter version of Theorem~\ref{thm:action-angle_intro}, for which the compact integral manifolds are $n$-dimensional Liouville tori, even on the exceptional hypersurface where the rank of $\Pi$ is only $2n-2$. In this sense, the result is similar to the symplectic case. Around a Liouville torus lying inside the exceptional hypersurface, there is a neighborhood $U$ with action-angle coordinates  $(\theta_1,\dots,\theta_n,\sigma_1,\dots,\sigma_{n}): U \to  \T^n \times B^n$ such that
\begin{equation}\label{pi_in_aacoo}
        \Pi=\sum_{i=1}^{n-1} \frbd{\theta_i}\wedge \frbd{\sigma_i} + \frac{1}{c} \sigma_n \frbd{\theta_n}\wedge \frbd{\sigma_n},
\end{equation}
and the ``action coordinates'' $\sigma_1,\ldots, \sigma_n$ depend only on the integrals of the system. The number $c\in \R$ is the modular period of the connected component of $Z$ in which the Liouville torus lies (see Definition~\ref{def:modperiod} below).

To see why we must assume that the exceptional hypersurface has a non-trivial normal bundle, consider the M\"obius band $\mathbb{R}^2 / (x, y) \sim (x+1, -y)$ with the $b$-Poisson structure $y\frbd{x}\wedge \frbd{y}$ and the function $f = -\textrm{log}|y|$. Here, the orbits of $X_f = \frbd{x}$ do not define a trivial fibration semilocally around the orbit $\{y = 0\}$, so there cannot be action-angle coordinates for this example. For the remainder of this paper, we will assume that the exceptional hypersurface has a trivial normal bundle. That is, it is cut out by a defining function on a tubular neighborhood of $Z$.

\begin{remark}
 In contrast to the general action-angle coordinate theorem for Poisson manifolds, where we only have action-angle coordinates on the points where $\Pi$ has full rank, our result specifically concerns the degeneracy locus of $\Pi$. Away from the degeneracy locus we can apply the standard action-angle coordinate theorem for symplectic manifolds.
\end{remark}

In the next section we will introduce the definitions needed to formulate the result above in a more precise way.


\subsection{Basics on $b$-symplectic geometry}

Recall from  \cite{Guillemin2012} that a $b$-manifold $(M, Z)$ is a manifold $M$ together with an embedded hypersurface $Z$; a $b$-vector field is a vector field on $M$ which is tangent to $Z$, and that the space of $b$-vector fields are naturally sections of a vector bundle on $M$ called the $b$-tangent bundle $^bTM$. Sections of the exterior algebra of the $b$-cotangent bundle $^bT^*M := (^bTM)^*$ are called $b$-forms, and there is a natural differential on the space of $b$-forms which defines a differential complex whose cohomology (called $b$-cohomology, written $^bH^*(M)$) satisfies the following relationship to the cohomology of $M$ and $Z$,

\begin{theorem}[{\bf{Mazzeo-Melrose}}]\label{thm:mazzeomelrose}
The $b$-cohomology groups of $M^{2n}$ satisfy
$$^b H^*(M)\cong H^*(M)\oplus H^{*-1}(Z).$$
\end{theorem}

Under the Mazzeo-Melrose isomorphism, a $b$-form of degree $p$ has two parts: its first summand, the \emph{smooth} part, is determined (by Poincar\'{e} duality) by integrating the form along any $p$-dimensional cycle transverse to $Z$ (such an integral is improper due to the singularity along $Z$, but the principal value of this integral is well-defined). The second summand, the \emph{singular} part, is the residue of the form along $Z$.

\begin{definition}({\bf $b$-functions})
 A $b$-function on a $b$-manifold $(M,Z)$ is a function which is smooth away from $Z$, and near $Z$ has the form
$$ c \log |t| + g, $$
where $c\in \R, g\in C^\infty,$ and $t$ is a local defining function. The sheaf of $b$-functions is denoted $^b \! C^\infty$.
\end{definition}

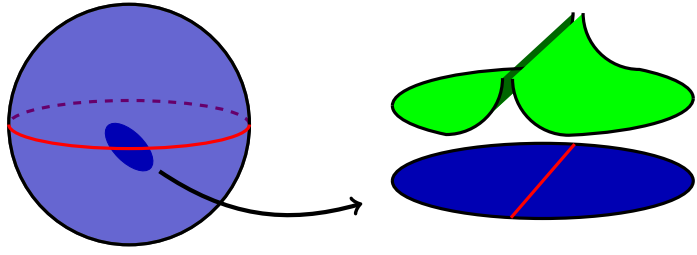
\begin{figure}
\centering
\begin{tikzpicture}

\pgfmathsetmacro{\ellbasex}{7}
\pgfmathsetmacro{\ellbasey}{3.5}

\pgfmathsetmacro{\majoraxis}{2}
\pgfmathsetmacro{\minoraxis}{.5}

\pgfmathsetmacro{\rlineymid}{4.25}
\def\R{1.6}
\pgfmathsetmacro{\circlex}{1.5}

\draw[dashed, very thick, color = red] (\circlex + \R, \rlineymid) arc (0:180:{\R} and {\R * .2});
\draw[very thick, fill = dblue, opacity = .6] (\circlex, \rlineymid) circle (\R);
\draw[very thick] (\circlex, \rlineymid) circle (\R);
\draw[rotate around={-45:(\circlex, \rlineymid - .3)},dblue, fill = dblue] (\circlex, \rlineymid - .3) ellipse (.4 and .2);
\draw[very thick, color = red] (\circlex + \R, \rlineymid) arc (0:-180:{\R} and {\R * .2});

\node (circaround) at (\circlex, \rlineymid - .3) [circle,draw=none,thick, minimum size=1cm] {};
												
\draw[ultra thick, ->] (circaround)  to [bend right = 25] (\ellbasex - 1.2*\majoraxis, \ellbasey - .3);

\draw[very thick, fill = dblue] (\ellbasex - \majoraxis, \ellbasey) arc (-180:180:{\majoraxis} and {\minoraxis});
\node (baseone) at (\ellbasex - \majoraxis * .27, \ellbasey - \minoraxis * 1.25) {};
\node (basetwo) at (\ellbasex + \majoraxis * .27, \ellbasey + \minoraxis * 1.25) {};
\draw[very thick, red] (baseone) -- (basetwo);

\draw[draw = none, fill = green] (\ellbasex - \majoraxis, \ellbasey + 1) arc (180:230:{\majoraxis} and {\minoraxis}) arc (-90:-47.4:.75) node (cylh) {} -- ++(0.95, 0.875) arc(-47.4:-90:.75) arc(100:180:{\majoraxis} and {\minoraxis}) -- cycle;
\draw[draw = none, fill = dgreen] (cylh)  arc (-47.4:0:.75) -- ++(0.95, 0.875) arc(0:-47.4:.75) -- cycle;

\draw[very thick] (\ellbasex - \majoraxis, \ellbasey + 1) arc (180:230:{\majoraxis} and {\minoraxis}) arc (-90:0:.75);
\draw[very thick] (\ellbasex - \majoraxis, \ellbasey + 1) arc (180:100:{\majoraxis} and {\minoraxis}) arc (-90:0:.75);

\draw[draw = none, fill = green] (\ellbasex + \majoraxis, \ellbasey + 1.1) arc (0:50:{\majoraxis} and {\minoraxis}) arc (270:180:.75) -- ++(-0.95, -0.875) arc(180:270:.75) arc (-80:0:{\majoraxis} and {\minoraxis}) -- cycle;

\draw[very thick] (\ellbasex +\majoraxis, \ellbasey + 1.1) arc (0:50:{\majoraxis} and {\minoraxis}) arc (270:180:.75);
\draw[very thick] (\ellbasex +\majoraxis, \ellbasey + 1.1) arc (0:-80:{\majoraxis} and {\minoraxis}) arc (270:180:.75);

\end{tikzpicture}
\caption{A $b$-function on a $b$-manifold with a logarithmic singularity near the exceptional hypersurface.} \label{fig:L1}
\end{figure}

In other words, the space of $b$-functions corresponds to the functions on $M \backslash Z$ whose differential extends to a globally-defined closed $b$-form on $M$ whose $b$-cohomology class has zero smooth part.

A closed $b$-form of degree two which is nondegenerate as a section of $\Lambda^2 ({^b}T^*M)$ is called a {\bf $b$-symplectic form}. Dualizing a $b$-symplectic form gives a $b$-Poisson bivector, and vice versa; the structures are equivalent (\cite{Guillemin2012}) and we will use the terms $b$-symplectic manifold and $b$-Poisson manifold interchangeably. For any $b$-function on a $b$-symplectic manifold $(M, \omega)$ there is a (smooth) vector field $X_{f_n}$ defined by $\iota_{X_{f_n}}\omega = -df_n$ which we call the $b${\bf -Hamiltonian vector field} of $f_n$. A $\mathbb{T}^k$ action on a $b$-symplectic manifold $(M^{2n}, \omega)$ is called \textbf{$b$-Hamiltonian} if the fundamental vector fields are the $b$-Hamiltonian vector fields of functions which Poisson commute. Such an action is called \textbf{toric} if $k = n$.

\paragraph{\bf The topology of $Z$}
For a given volume form $\Omega$ on a Poisson manifold $M$ the associated {\bf modular vector field} $u_{\mod}^\Omega$ is defined as the following derivation:
 \[
C^\infty(M) \to \R : \, f\mapsto \frac{\SL_{X_f}\Omega}{\Omega}.
\]
It can be shown (see for instance \cite{Weinstein2}) that this is indeed a derivation and, moreover, a Poisson vector field. Furthermore, for different choices of volume form $\Omega$, the resulting vector fields only differ by a Hamiltonian vector field.

The topology of the exceptional hypersurface $Z$ of a $b$-symplectic structure has been studied in \cite{Guillemin2011} and \cite{Guillemin2012}. In \cite{Guillemin2011} it was shown that if $Z$ is compact and connected, then it is the mapping torus of any of its symplectic leaves $\SL$ by the flow of the any choice of modular vector field $u$:
\[
 Z = (\SL \times [0,k])/_{(x,0)\sim (\phi(x),k)},
\]
where $k$ is a certain positive real number and $\phi$ is the time-$k$ flow of $u$. In particular, all the symplectic leaves inside $Z$ are symplectomorphic.

In the transverse direction to the symplectic leaves, all the modular vector fields flow with the same speed. This allows the following definition:

\begin{definition}[Modular period]\label{def:modperiod}
 Taking any modular vector field $u_{\mod}^\Omega$, the {\bf modular period} of $Z$ is the number $k$ such that $Z$ is the mapping torus
$$ Z = (\SL \times [0,k])/_{(x,0)\sim (\phi(x),k)},$$
and the time-$t$ flow of  $u_{\mod}^\Omega$ is translation by $t$ in the $[0, k]$ factor above.
\end{definition}

\subsection{Basics on KAM theory}\label{sec:kam}

\begin{figure}[!b]
\begin{center} Perturbation
\includegraphics[scale=0.55]{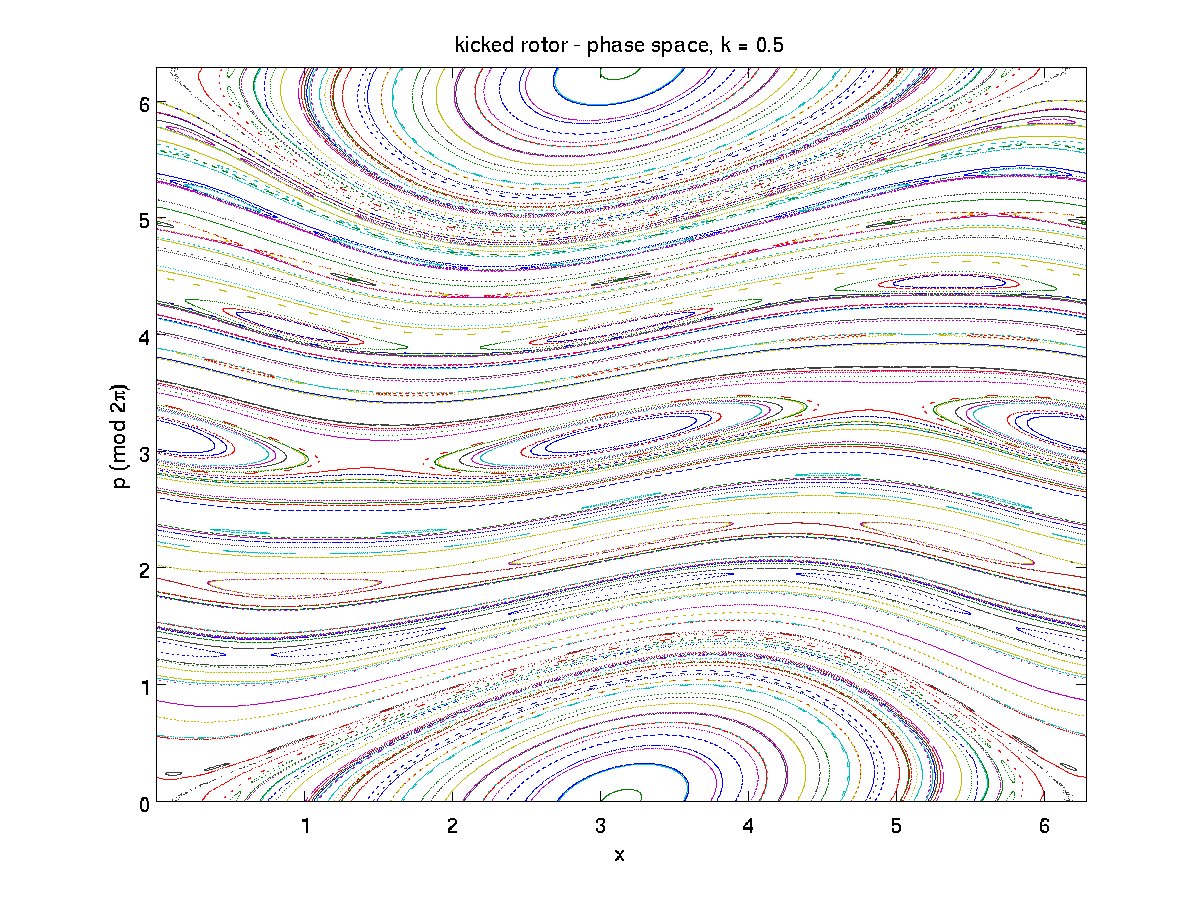}
\end{center}
\caption{The phase space of the kicked rotor: This system is a discrete approximation of the perturbed system arising from the mathematical pendulum (phase space $\S^1 \times \R$) to which a small periodic t-dependent vertical force is added. The variable $x \in \S^1$ is the angle coordinate, the variable $p \in \R$ is the conjugated momentum. The parameter $k$ measures the size of the perturbation.}
\end{figure}

The classical KAM theorem -- named after Kolmogorov, Arnold and Moser -- is a statement about the stability of integrable Hamiltonian systems on symplectic manifolds: Roughly speaking, it implies that ``most'' Liouville tori of an integrable system persist under sufficiently small perturbations of the Hamiltonian function of the system. In physics, the Hamiltonian function $H$ (or ``Hamiltonian'' for short) is a function on a symplectic manifold $M$ which determines the evolution of any other function $g$ on $M$ via the equation
$$ \dot g = \{g,H\}.$$
The pair of a symplectic manifold and a Hamiltonian function is called a {\bf Hamiltonian (dynamical) system}. \\

An integrable system on a symplectic manifold $M$ together with a Hamiltonian function is called an {\bf integrable Hamiltonian system} if $\{f_i,H\}=0$ for all integrals $f_i$. Physically, the last condition corresponds to the fact that the integrals represent constants of motion. From the action-angle coordinate theorem we know that the manifold is semilocally around a Liouville torus symplectomorphic to the product $\T^n \times B^n$ with the standard symplectic structure, with coordinates $(\varphi, y)$, and with the integrals $\{y_1, \dots, y_n\}$ being the components of the projection to the $B^n$ component. Restricting to such a neighborhood, we can write their equations of motion explicitly:
\begin{align}\label{ham_sys}
\begin{split}
    \dot { \varphi} &= \frbd{ y} H ( \varphi, y),\\
    \dot { y} &= -\frbd{ \varphi} H ( \varphi, y).
 \end{split}
\end{align}
Here, the expressions in both lines are maps from $M$ to $\R^n$. \\

In particular, if we consider the requirement $\{y_i,H\}=0$ means that $H$ is independent of the angle coordinates $\varphi$. Therefore, we can write
$$H(\varphi, y) = h(y)$$
for some function $h:B^n \to \R$. This system evolves in a very simple way:

\begin{align}\label{ham_sys_aa_coo}
    \begin{split}
       &\dot { \varphi} = \frbd{ y}H( \varphi, y)=: \omega(y)\\
       &\dot {y} = 0.
    \end{split}
\end{align}
and has solutions of the form
\begin{equation}\label{sol_aa}
 y(t) = y_0,\qquad \varphi(t)= {\varphi_0} +  \omega (y_0) t,
\end{equation}
i.e. the angle coordinates wrap around the torus $\T^n \times \{y_0\}$ in a linear fashion. The components of $\omega(y_0)$ are {\it rationally dependent} if
$$ \omega \cdot k := \sum_{i=1}^n \omega_i k_i = 0 \qquad \text{for some }k\in \Z^n \without{0} .$$
Otherwise, if the components of $\omega(y_0)$ are rationally independent, the motion is called {\it quasi-periodic} and the resulting trajectory fills the torus densely. In KAM theory a particular kind of rationally independent frequency vectors are of interest --- so-called {\it strongly non-resonant} or {\it Diophantine} frequency vectors:

\begin{definition}[Diophantine condition]
 An $n$-tuple $\omega \in \R^n$  is called {\it Diophantine} if there exist $L,\gamma >0$ such that
\begin{equation}\label{cond_dioph}
 |  \omega \cdot k|  \ge  \frac{L}{|k|^\gamma}\qquad \text{for all } k \in \Z^n\without{0},
\end{equation}
where $|k|:= \sum_{i=1}^n |k_i|$.
\end{definition}

Not only do rationally dependent resp. independent frequency vectors result in very different kinds of motion along the tori, they also behave very differently under {\it perturbations} of the Hamiltonian system. Indeed, it can be shown that the tori corresponding to rationally dependent frequency vectors -- so-called resonant tori -- are generically destroyed by arbitrarily small perturbations. In contrast, strongly non-resonant tori survive under sufficiently small perturbations. That is, there is a symplectomorphism on a neighborhood of the torus taking the perturbed trajectory to the linear flow on a torus with unchanged frequency vector. The precise conditions are stated in the following theorem:

\begin{theorem}[KAM]\label{th:kam}
  Let $H( \varphi,  y) = h( y)$ be an analytic function on $\T^n \times B^n$ with frequency map $ \omega( y):= \frbd{ y} h( {y})$. If $ {y_0} \in B^n$ has Diophantine frequency vector $ \omega := \omega(y_0)$ and if the \emph{non-degeneracy condition} holds:
\begin{equation}\label{nondegenerate}
 \det \frbd{ y}  \omega( {y_0}) \neq 0,
\end{equation}
then the torus $\T^n \times \{ {y_0}\}$  persists under sufficiently small perturbations of $H$. That is, if $P$ is any analytic function on $\T^n \times B^n$ and $\epsilon>0$ sufficiently small, the perturbed system
  $$ H_\epsilon =H + \epsilon P$$
admits an invariant torus $\mathcal{T}$ close to $\T^n \times \{ {y_0}\}$.

Moreover, the flow $\gamma^t$  of the perturbed system on $\mathcal{T}$ is conjugated via a diffeomorphism $\psi:\T^N\to \mathcal{T}$ to the linear flow with frequency vector $ \omega$ on $\T^n$, i.e.
$$\psi^{-1} \circ \gamma^t \circ \psi ( {\varphi_0}) =  {\varphi_0} +  \omega t.$$
\end{theorem}

 The basis is Kolmogorov's theorem, which we state below and whose proof is the heart of KAM theory. It tells us that we can ``correct'' a sufficiently small perturbation of a certain type of Hamiltonian via a symplectomorphism close to the identity. The KAM Theorem follows as an easy corollary. Before stating Kolmogorov's Theorem, we introduce the following

\begin{definition}[Kolmogorov normal form]\label{def:kolm_normal_form}
 Let $H:\T^n \times B^n \to \R$ be an analytic function. We say that $H$ is in {\bf Kolmogorov normal form} (with frequency vector $\omega$) if it is of the form
\begin{equation}\label{kolm_normal_form}
 H( \varphi, y)= E + \omega \cdot y + Q(\varphi, y)
\end{equation}
for some $E\in \R$, $\omega\in \R^n$ and a function $Q:\T^n \times B^n \to \R$ which is quadratic in $y$, meaning that
$$ Q( \varphi, 0)= 0,\quad \frbd{y} Q(\varphi, 0)= 0.$$
The Kolmogorov normal form is called {\bf non-degenerate} if
$$\det \frac{1}{(2\pi)^n} \int_{\T^n}  \partial^2_{y} Q( \varphi, 0) d\varphi =: \det \langle \partial^2_{y} Q(\cdot, 0) \rangle \neq 0. $$
\end{definition}

\begin{theorem}[Kolmogorov]\label{thm:kolmogorov}
 Let $H:\T^n \times B_r^n \to \R$ be in non-degenerate Kolmogorov normal form~\eqref{kolm_normal_form} with Diophantine frequency vector $ \omega$.
Consider any (analytic) perturbation $P:\T^n \times B_r^n \to \R$. Then there exists $\epsilon_0$ such that for all $0<\epsilon<\epsilon_0$ there exists a symplectomorphism
$$\psi: \T^n \times B^n_{r_\ast} \to \T^n \times B_r^n $$
for some $0<r_\ast<r$ which transforms the perturbed Hamiltonian $ H_\epsilon:= H + \epsilon P $
into Kolmogorov normal form:
$$(H_\epsilon \circ \psi) ( \varphi,  y) = E_\ast +  \omega \cdot  y + Q_\ast( \varphi,  y)$$
where\footnote{For functions on $\T^n \times B^n$ the norm $\norm{\cdot}$ denotes the supremum norm; for maps $\T^n \times B^n \to \T^n \times B^n$ we regard the target space as a subset of $\R^{2n} \times \R^n$ and use the supremum norm with respect to the Euclidean norm.}
$\norm{\psi - \id}, |E_\ast - E|$ and $\norm{Q- Q_\ast}$ are of order $\epsilon$. In particular $\psi$ is $\epsilon$-close to the identity, i.e. there exists a constant $C$ such that $\norm{\psi - \id} < C \epsilon$. \\
\end{theorem}

\begin{remark}
  A concrete choice of $\epsilon_0$ in Theorem \ref{thm:kolmogorov} can be given, but it is not important for the purpose of this paper.
  \end{remark}
  \begin{remark} Note that, in the new Kolmogorov normal form $K_\epsilon$, the constant and quadratic parts $E_\ast$ resp. $Q_\ast$ are slightly different from the original $E$ resp. $Q$, but the frequency vector $\omega$ is the same.

\end{remark}

In Section \ref{sec:kamb} we will present a KAM result for the $b$-symplectic case.



\section{Integrable systems on $b$-symplectic manifolds}
An integrable system on a $b$-symplectic manifold $(M^{2n}, \omega)$, defined according to the standard definition \ref{def:intsyspoisson} for Poisson manifolds, will only define a distribution of rank at most $2n-2$ on the exceptional hypersurface $Z$. The vector field $\frbd{\theta_n}$ in the expression~\eqref{pi_in_aacoo} is actually not a Hamiltonian vector field of the last angle coordinate $\sigma_n$ (which is a defining function of $Z$), but it \emph{is} a $b$-Hamiltonian vector field of $\log |\sigma_n|$. To obtain an action-angle coordinate theorem for integrable systems on a $b$-symplectic manifold as in Equation~\eqref{pi_in_aacoo}, we allow our integrals of motion to be $b$-functions.
\begin{definition}[{\bf $b$-integrable system}]\label{def:pbintegrable}  A  $b$-integrable system on a $2n$-dimensional $b$-symplectic manifold $(M^{2n},\omega)$ is a set of  $n$ pairwise Poisson commuting $b$-functions $F=(f_1,\ldots,f_{n-1},f_n)$ (i.e, $\{f_i,f_j\}=0$),  satisfying,
$df_1 \wedge \dots \wedge df_n$ is nonzero as a section of $\wedge^n  ({^b} T^*(M))$ on a dense subset of $M$ and on a dense subset of $Z$. We say that a point in $M$ is {\bf regular} if the vector fields $X_{f_1}, \dots, X_{f_n}$ are linearly independent (as \emph{smooth} vector fields) at it. \end{definition}

Notice that if a point on $Z$ is regular, then at least one of the $f_i$ must be non-smooth there.

On the set of regular points, the distribution given by $X_{f_1},\ldots, X_{f_n}$ defines a foliation $\SF$. We denote the integral manifold through a regular point $m\in M$ by $\SF_m$. If $\SF_m$ is compact, then it is an $n$-dimensional torus (also referred to as ``(standard) {\bf Liouville torus}''). Because the $X_{f_i}$ are b-vector fields and are therefore tangent to $Z$, any Liouville torus that intersects $Z$ actually lies inside $Z$.
Two ($b$-)integrable systems $F$ and $F'$ are called {\bf equivalent} if there is a map $\mu:\R^n \supset F(M) \to \R^n$ taking one system to the other: $F' = \mu \circ F$. We will not distinguish between equivalent integrable systems.
\begin{remark}
Near a regular point of $Z$, a $b$-integrable system on a $b$-symplectic manifold is equivalent to one of the type $F=(f_1,\ldots,f_{n-1},f_n)$, where $f_1,\ldots, f_{n-1}$ are $C^\infty$ functions and $f_n$ is a $b$-function. In fact, we may always assume that $f_n = c\log|t|$, where $c \in \mathbb{R}$ and $t$ is a global defining function for $Z$.
\end{remark}
\subsection{Examples of $b$-integrable systems on $b$-Poisson manifolds}\subsubsection{Integrable systems on manifolds with boundary}\label{manifoldswithboundary} Given a $b$-symplectic manifold $(M, \omega)$ and a component $W$ of $M \backslash Z$, $\restr{\omega}{W}$ is a classic symplectic form on $W$. The closure of $W$ is a manifold with boundary, and the asymptotics of $\restr{\omega}{W}$ near this boundary can been described in the following way, following \cite{NestandTsygan}. For each $p \in \partial W$, there is a neighborhood diffeomorphic to $\{(x_1, y_1, \dots, x_n, y_n) \mid x_1 \geq 0\}$ on which $\omega = \frac{1}{x_1}dx_1 \wedge dy_1 + \sum_{i > 1} dx_i \wedge dy_i$.

This observation enables us to study manifolds with boundary equipped with a symplectic forms in its interior, as long as these symplectic forms have the kind of asymptotics described above. By taking the double of such a manifold with boundary, the boundary becomes a hypersurface, and the symplectic form extends to a $b$-symplectic form on the double, and we may then apply the results of this paper.

For example, consider the upper hemisphere (including the equator) $H_+$ of $S^2$. The symplectic form $\frac{1}{h} dh \wedge d\theta$ defined on the interior of $H_+$ extends to a $b$-symplectic form on the double of $H_+$, which is $S^2$. This example can be generalized to $H_+ \times M$, where $M$ is any symplectic manifold, endowed with the product symplectic structure. Moreover, any integrable system on the interior of $H_+ \times M$ which has asymptotics compatible with those of $\omega$ near the boundary extends to a $b$-integrable system on the double $S^2 \times M$. For example, if $(f_1, \dots, f_n)$ is an integrable system on $M$, then the integrable system $(\log|h|, f_1, \dots, f_n )$ on $H_+ \times M$ extends to the $b$-integrable system $(\log|h|, f_1, \dots, f_n )$ on $S^2 \times M$.

\subsubsection{Restricted 3-body problem}
The restricted 3-body problem\footnote{Special thanks to Amadeu Delshams for explainig us the Elliptic Restricted Three Body Problem and revealing its connection to our Poisson structures.} is a simplified version of the general 3-body problem where one of the bodies is assumed to have negligible mass
(an ``asteroid'' or ``comet'' or ``small planet'') and hence the two other bodies, the primaries (``star and big planet''), move independently of it
according to Kepler's laws for the 2-body problem. Here we consider the planar restricted 3-body problem, where the massless body
lies in the same plane as the other two.  The time-dependent self-potential of the small body  is then given by
$$U(q,t)= \frac{1-\mu}{|q-q_1|} + \frac{\mu}{|q-q_2|},$$
where we assume that the masses of the star and big planet are normalized, $q_1=q_1(t)$ is the position of the star (``Sun'') with mass $1-\mu$
at time $t$ and $q_2=q_2(t)$ the position of the planet (``Jupiter'') with mass $\mu$. Typically it is assumed that the primaries
revolve along circles centered in their center of masses (the Circular Restricted Three Body Problem), and the dependence
with respect to time of the potential can be removed in a rotating frame of coordinates. In general, the primaries
move around their center of mass on elliptic orbits with some non-zero eccentricity (the Elliptic Restricted Three Body Problem).

The Hamiltonian of the system is given by
$$H(q,p,t)= p^2/2 - U(q,t),\quad (q,p) \in \R^2 \times \R^2,$$
where $p=\dot q$ is the momentum of the planet. \\

Because of the rotational symmetry of the problem, it is natural to introduce polar coordinates $(r,\alpha)$ to describe the problem.
For $q=(X,Y)\in \R^2\without{0}$  we consider polar coordinates the  $(r,\alpha) \in \R^+ \times \T$ (
$ X = r \cos \alpha, Y = r \sin \alpha$). We transform the momentum variables $p=(P_X, P_Y)$ in such a way that the change of coordinates
$$(X,Y, P_X, P_Y) \mapsto (r,\alpha, P_r=:y, P_\alpha=:G)$$
is canonical, i.e. the symplectic structure remains the same.

To study the behaviour at infinity, it is standard to introduce the so-called {\bf McGehee coordinates}
$(x,\alpha,y,G)$, where $$ r=\frac{2}{x^2}, \quad x \in \R^{+},$$
which can be then extended to the infinity  $x=0$.

This transformation is non-canonical (and because of this in the classical literature in celestial mechanics
the symplectic structure is then totally abandoned; an exception is~\cite{delshams}),
i.e. the symplectic structure changes and one can check that it is given for $x>0$ by
$$-\frac{4}{x^3} dx \wedge dy + d\alpha \wedge d G.$$
This extends naturally to a $b^3$-symplectic structure in the sense of \cite{scott} on $\R\times \T \times \R^2$.
The integrable two body problem for $\mu=0$ (`star plus asteroid'') on the initial system of coordinates,
where action-angle variables are given by the so-called Delaunay coordinates
(which give rise to a degenerate Hamiltonian which only depends essentially on one action),
gets transformed to a $b^3$-integrable system with the new structure.

\section{Action-angle  variables for $b$-Poisson manifolds}\label{sec:aacoo}

The goal of this section is to prove the existence of action-angle variables for $b$-Poisson manifolds (Theorem~\ref{thm:aa} below). In the following subsections, let $(M,\omega,F)$ be a $b$-integrable system on the $b$-symplectic manifold $(M,\omega)$ with critical hypersurface $Z$. We start by establishing a local normal form result for integrable systems.


\subsection{Darboux-Carath\'{e}odory Theorem}

The following result \emph {locally} extends a set of $n$ Poisson commuting functionally independent $b$-functions to a ``$b$-Darboux'' coordinate system:

\begin{theorem}\label{th:darboux-caratheodory}
 Let $(M^{2n}, \omega)$ be a $b$-symplectic manifold, $m$ be a point on the exceptional hypersurface $Z$, and $f_1,\ldots,f_{n}$ be $b$-functions, defined on a neighborhood of $m$, which Poisson commute and have the property that $X_{f_1}\wedge \dots \wedge X_{f_n}$ is a nonzero section of $\Lambda^n(^bTM)$ at $m$. Then there exist $b$-functions $(q_1, \dots, q_n)$ around $m$ such that
\[ \omega = \sum_{i=1}^{n} df_i \wedge dq_i.\]
and the vector fields $\{X_{f_i}, X_{q_j}\}_{i, j}$ commute.

Moreover, if $f_n$ is not a smooth function, i.e. $f_n = c\log|t|$ for some $c \neq 0$ and some local defining function $t$ of $Z$, then the $q_i$ functions can be chosen to be smooth functions for which $(f_1, \dots, f_{n-1}, t, q_1, \dots, q_n)$ are local coordinates around $m$.
\end{theorem}

\begin{proof}
For this proof, we will adopt the notation that for a $1$-form $\mu$, the vector field $X_{\mu}$ is the vector field satisfying $\iota_{X_{\mu}}\omega = -\mu$. We begin by inductively constructing a collection $\{\mu_1, \dots, \mu_n\}$ of 1-forms with the property that
\begin{align*}
\mu_i(X_{f_j}) &= \delta_i^j\\
\mu_i(X_{\mu_j}) &= 0 \hspace{1cm} \textrm{for} \ j < i.
\end{align*}
Assume that we have successfully constructed $\mu_j$ for $j<i$, and moreover that this construction satisfies
\[
0 \neq X_{f_1} \wedge \dots, X_{f_n} \wedge X_{\mu_1} \wedge \dots \wedge X_{\mu_{i-1}}
\]
(as a section of $\Lambda^*({^b}TM)$) locally near $m$, and consider the problem of constructing $\mu_i$. Let $P_i$ and $K_i$ be the subbundles (of ${^b}TM$ and ${^b}T^*M$ respectively) defined by
\begin{align*}
P_i &= \textrm{span}(\{X_{f_j}, X_{\mu_k}\}_{j \neq i, k < i})\\
K_i &= \textrm{ker}(P_i)
\end{align*}
The bundle $K_i$ is a codimension-$(n + i - 2)$ subbundle of $^bT^*M$. By the inductive hypothesis, $X_{f_i}$ is not contained in $P_i$, so there is a section $\mu_i$ of $K_i$ for which $\mu_i(X_{f_i}) = 1$. The fact that $P_i$ and $X_{f_i}$ are in the kernel of $df_i$, but $X_{\mu_i}$ is not, reveals that $X_{\mu_i} \notin \textrm{span}(P_i \cup X_{f_i})$, so
\[
0 \neq X_{f_1} \wedge \dots \wedge X_{f_n} \wedge X_{\mu_1} \wedge \dots \wedge X_{\mu_{i}},
\]
completing the induction.

Because the $df_i$ and $\mu_i$ are $b$-forms, the $X_{f_i}$ and $X_{\mu_i}$ are $b$-vector fields. Using the fact that the $f_i$ functions Poisson commute, and the properties of the $\mu_i$ constructed above,
\[
\left( \begin{array}{c c} \omega(X_{f_i}, X_{f_j}) & \omega(X_{\mu_i}, X_{f_j}) \\ \omega(X_{f_i}, X_{\mu_j}) & \omega(X_{\mu_i}, X_{\mu_j}) \end{array} \right) = \left( \begin{array}{c c} 0 & -I \\ I & 0 \end{array} \right),
\]
so $\omega = \sum_{i = 1}^n df_i \wedge \mu_i$. To check that the $\mu_i$ are closed (and therefore exact) in a neighborhood of $m$, we apply the Cartan formula for the exterior derivative to calculate that
\begin{align*}
d\mu_i (X_{f_j}, X_{f_k}) &= X_{f_j}(\mu_i(X_{f_k})) - X_{f_k}\mu_i(X_{f_j}) - \mu_i([X_{f_j}, X_{f_k}])\\
&= X_{f_j} \delta_i^k - X_{f_k} \delta_i^j - \mu_i(0) = 0\\
d\mu_i (X_{f_j}, X_{\mu_k}) &= X_{f_j}\mu_i(X_{\mu_k}) - X_{\mu_k}\mu_i(X_{f_j}) - \mu_i([X_{f_j}, X_{\mu_k}])\\
&= X_{f_j} (0) - X_{\mu_k} \delta_i^j - \mu_i(0) = 0\\
d\mu_i (X_{\mu_j}, X_{\mu_k}) &= X_{\mu_j}\mu_i(X_{\mu_k}) - X_{\mu_k}\mu_i(X_{\mu_j}) - \mu_i([X_{\mu_j}, X_{\mu_k}])\\
&= X_{\mu_j} (0) - X_{\mu_k} (0) - \mu_i(0) = 0.\\
\end{align*}
Letting $q_i$ be any local primitive of $\mu_i$ yields the first part of the result.

In the case when $f_n = c\log|t|$ is non-smooth, we can modify our inductive construction of the $\mu_i$ so that in addition to requiring that the $\mu_i$ be in $K_i$, we also insist that they be in $T^*M \subseteq {^bT^*M}$. To check that this restriction is consistent with the requirement that $\mu_i(X_{f_i}) = 1$, we must check that the kernel of $X_{f_i}$ is not identically equal to $T^*M$, i.e. $X_{f_i}$ does not vanish at $m$ when viewed as a section of $TM$. But this is clear from the fact that $X_{f_i}$ does not vanish at $m$ when viewed as a section of ${^b}TM$, and $0 = \{f_n, f_i\} = \frac{c dt}{t}(X_{f_i})$.\\

This proves that the $q_i$ can be chosen to be smooth functions. The fact that $\{X_{f_1}, \dots, X_{f_{n-1}}, X_t, X_{q_1}, \dots, X_{q_n}\}$ pairwise commute follows from the fact that $\{X_{f_i}, X_{q_j}\}_{i, j}$ do, so $(f_1, \dots, f_{n-1}, t, q_1, \dots, q_n)$ indeed are local coordinates.

\end{proof}

\subsection{Foliation by Liouville tori}

Suppose $m$ is a regular point of a $b$-integrable system and the Liouville torus is contained in the exceptional hypersurface. The following theorem shows that the $b$-integrable system is isomorphic to one of the form
\begin{equation}\label{eq:standardform}
(\T^n \times B^n,\omega,F :=(f_1,\ldots,f_{n-1}, f_n = \log|t|)),
\end{equation}
Here, ``isomorphic'' means that there is a $b$-symplectomorphism of a neighborhood of the torus with the $b$-manifold above that takes the integrals of motion to $F$.

\begin{proposition}\label{prop:standarddonut}
Let $m\in Z$ be a regular point of a $b$-integrable system $(M,\omega, F)$. Assume that the integral manifold $\SF_m$ through $m$ is compact (i.e. a torus $\T^n$). Then there exists a neighborhood $U$ of $\SF_m$ and a diffeomorphism
$$\phi:U\simeq \T^n \times B^{n},$$
which takes the foliation $\SF$ to the trivial foliation $\{\T^n\times \{b \}\}_{b\in B^n}$.
\end{proposition}

\begin{proof}
The proof is the same as in~\cite{Laurent-Gengoux2010} (Prop. 3.2), the only difference being that the last integral $f_n$ in our system $F=(f_1,\ldots, f_n)$ is a $b$-function, $f_n = \log|t|$ for some $t$. However, the foliation given by the ($b$-)Hamiltonian vector fields of $F$ is the same as the one given by the level sets of the smooth function $\tilde F:=  (f_1,\ldots,f_{n-1}, t)$. Then, as in~\cite{Laurent-Gengoux2010}, choosing an arbitrary Riemannian metric on $M$ defines a canonical projection $\psi: U \to \SF_m$. Setting  $\phi:= \psi \times \tilde F$ we have a commuting diagram
\begin{align}
  \begin{diagram}
    \node{U}\arrow{e,t,--}{\phi}\arrow{se,b}{\tilde F}\node{\T^n \times B^n}\arrow{s,r}{p}\\
    \node[2]{B^{n}}
  \end{diagram}
\end{align}
which provides the necessary isomorphism of $b$-integrable systems.
 \end{proof}

\vspace{3mm}
\begin{figure}
\centering

\begin{tikzpicture}

\pgfmathsetmacro{\sizer}{1.2}
\pgfmathsetmacro{\basept}{5}	

\pgfmathsetmacro{\xone}{0}
\pgfmathsetmacro{\xtwo}{1.5}
\pgfmathsetmacro{\xthree}{3}
\pgfmathsetmacro{\xfour}{4.5}
\pgfmathsetmacro{\xfive}{6}

\pgfmathsetmacro{\ymid}{4.25}
\pgfmathsetmacro{\ytop}{5.65}
\pgfmathsetmacro{\ybottom}{2.85}

\DrawFilledDonutops{\xone}{\ymid}{.55*\sizer}{1.2*\sizer}{-90}{dblue}{very thick}{white}
\DrawFilledDonutops{\xtwo}{\ymid}{.55*\sizer}{1.2*\sizer}{-90}{dblue}{very thick}{white}
\DrawFilledDonutops{\xthree}{\ymid}{.55*\sizer}{1.2*\sizer}{-90}{dblue}{very thick}{white}
\DrawFilledDonutops{\xfour}{\ymid}{.55*\sizer}{1.2*\sizer}{-90}{dblue}{very thick}{white}
\DrawFilledDonutops{\xfive}{\ymid}{.55*\sizer}{1.2*\sizer}{-90}{dblue}{very thick}{white}
	
\DrawDonut{\xone}{\ymid}{.55*\sizer}{1.2*\sizer}{-90}{black}{very thick}
\DrawDonut{\xtwo}{\ymid}{.55*\sizer}{1.2*\sizer}{-90}{black}{very thick}
\DrawDonut{\xthree}{\ymid}{.55*\sizer}{1.2*\sizer}{-90}{black}{very thick}
\DrawDonut{\xfour}{\ymid}{.55*\sizer}{1.2*\sizer}{-90}{black}{very thick}
\DrawDonut{\xfive}{\ymid}{.55*\sizer}{1.2*\sizer}{-90}{black}{very thick}

\draw[very thick, ->] (\xone, \ybottom - 0.3) -- +(0, -1);
\draw[very thick, ->] (\xtwo, \ybottom - 0.3) -- +(0, -1);
\draw[very thick, ->] (\xthree, \ybottom - 0.3) -- +(0, -1);
\draw[very thick, ->] (\xfour, \ybottom - 0.3) -- +(0, -1);
\draw[very thick, ->] (\xfive, \ybottom - 0.3) -- +(0, -1);

\draw[ultra thick] (\xone - 1, \ybottom - 1.5) -- (\xfive + 1, \ybottom - 1.5);

\end{tikzpicture}
\caption{Fibration by Liouville tori}
\end{figure}
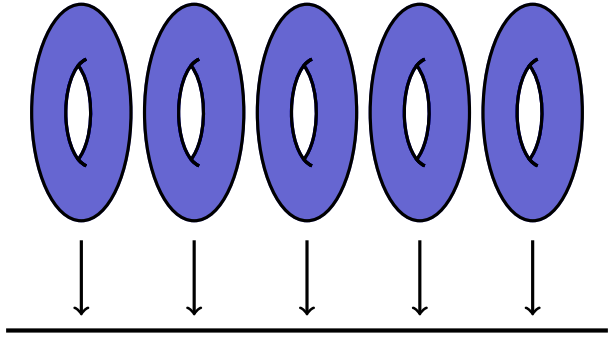

\vspace{3mm}

\subsection{The statement and the proof of the action-angle theorem}

\begin{theorem}\label{thm:aa}
 Let $(M, \omega, F = (f_1, \dots, f_{n-1}, f_n = \log|t|))$ be a $b$-integrable system, and let $m \in Z$ be a regular point for which the integral manifold containing $m$ is compact, i.e. a Liouville torus $\SF_m$. Then there exists an open neighborhood $U$ of the torus $\SF_m$ and coordinates $(\theta_1,\dots,\theta_n,\sigma_1,\dots,\sigma_{n}): U \to  \T^n \times B^n$ such that
\begin{equation}
        \omega|_U =\sum_{i=1}^{n-1} d\sigma_i \wedge d\theta_i  + \frac{c}{\sigma_n} d\sigma_n \wedge d\theta_n,
\end{equation}
where the coordinates $\sigma_1,\dots,\sigma_n$ depend only on $F$ and the number $c$ is the modular period of the component of $Z$ containing $m$.
\end{theorem}


By Proposition (\ref{prop:standarddonut}), we may assume our $b$-integrable system is in the form in equation (\ref{eq:standardform}), where $f_1, \dots, f_{n-1}$ are smooth. Although the vector fields $X_{f_1}, \dots, X_{f_n}$ define a torus action on each of the Liouville tori, $\T^n \times \{b\}_{b \in B^n}$, individually, it is not guaranteed that their flow defines a torus action on all of $\T^n \times B^n$. Our next goal is to construct an equivalent $b$-integrable system whose fundamental vector fields \emph{do} define a $\T^n$ action on a neighborhood of $\T^n \times \{0\}$.

\paragraph{\emph{Step 1 -- uniformization of periods (constructing a $\T^n$ action).}}

Denote by $\Phi_{X_{f}}^{s}$ the time-$s$ flow of the ($b$-)Hamiltonian vector fields $X_{f}$ and
consider the combined flow of the ($b$-)Hamiltonian vector fields $X_{f_1}, \dots, X_{f_n}$,
\begin{align*}
\Phi :\R^n \times ( \T^n \times B^n ) &\to ( \T^n \times B^n )\\
 \big((s_1,\ldots,s_n),(x, b)\big) &\mapsto \Phi_{X_{f_1}}^{s_1}\circ\dots\circ \Phi_{X_{f_n}}^{s_n}((x, b)).
\end{align*}

Because the $X_{f_i}$ are complete and commute with one another, this defines an $\R^n$-action on $\T^n \times B^n$. When restricted to a single orbit $\T^n \times \{b\}$ for some $b \in B^n$, the kernel of this action is a discrete subgroup of $\R^n$, hence a lattice $\Lambda_{b}$, called the {\bf period lattice} of the orbit $\T_n \times \{b\}$. Since the orbit is compact, the rank of $\Lambda_b$ is $n$. \\

The lattice $\Lambda_b$ will in general depend on $b$. The process of \emph{uniformization} entails modifying the action so that $\Lambda_b = \mathbb{Z}^n$ for all $b$. For any $b \in B^{n-1}\times\{0\}$ and any $a_i \in \mathbb{R}$, the vector field $\sum a_i X_{f_i}$ on $\T_n \times \{b\}$ is the $b$-hamiltonian vector field of $a_n\log|t| + \sum_{i=1}^{n-1} a_i f_i$, where $f_i$ are smooth for $i < n$. However, by Proposition 4 of \cite{Guillemin2013}, if such a vector field is 1-periodic, then $a_n = \pm c$, where $c$ is the modular period of the component of $Z$ containing $m$. Therefore, for all $b \in B^{n-1} \times \{0\}$, the lattice $\Lambda_b$ is contained in $\mathbb{R}^{n-1} \times c\mathbb{Z} \subseteq \mathbb{R}^n$. To perform the uniformization, pick smooth functions
$$(\lambda_1, \lambda_2, \dots, \lambda_n): B^n \rightarrow \mathbb{R}^n$$
such that
\begin{itemize}
\item $(\lambda_1(b), \lambda_2(b), \dots, \lambda_n(b))$ is a basis for the period lattice $\Lambda_b$ for all $b \in B^n$
\item $\lambda_i^n$ vanishes along $B^{n-1}\times \{0\}$ for $i < n$, and $\lambda_i^n$ equals the modular period $c$ along $B^{n-1}\times\{0\}$. Here, $\lambda_i^{j}$ denotes the $j^{\textrm{th}}$ component of $\lambda_i$.
\end{itemize}
 Such functions $\lambda_i$ exist that satisfy the first condition (perhaps after shrinking $B^n$) by the implicit function theorem, using the fact that the Jacobian of the equation $\Phi(\lambda, m) = m$ is regular with respect to the $s$ variables. The fact that they can be chosen to satisfy the second condition is due to the discussion above.

Using these functions $\lambda_i$ we define the ``uniformized'' flow
\begin{align*}
\tilde \Phi :\R^n \times ( \T^n \times B^n ) &\to ( \T^n \times B^n )\\
 \big((s_1,\ldots,s_n),(x, b)\big) &\mapsto \Phi\big(\sum_{i=1}^n s_i \lambda_i(c), (x,b) \big).
\end{align*}
The period lattice of this $\R^n$ action is constant now (namely $\Z^n$) and hence the action naturally defines a $\T^n$ action.

\paragraph{\emph{Step 2 -- the $\T^n$ action is $b$-Hamiltonian.}}

We want to find $b$-functions $\sigma_1, \dots, \sigma_n$

 such that the $b$-hamiltonian vector fields $X_{\sigma_i}$ are precisely the fundamental vector fields of the $\mathbb{T}^n$ action we constructed above, $Y_i = \sum_{j = 1}^n \lambda_i^j X_{f_j}$.

The Cartan formula for $b$-symplectic forms (where the differential is the one of the complex of $b$-forms \cite{Guillemin2012}  \footnote{The decomposition of a $k$-form in the $b$-complex as $\omega=\frac{dt}{t}\wedge \alpha+\beta$ for  $\alpha, \beta$ De Rham forms proved in \cite{Guillemin2012} allows to extend the Cartan formula to the setting of $b$-forms}) gives
\begin{align} \label{eqn:poisson}
\mathcal{L}_{Y_i}\mathcal{L}_{Y_i}\omega &= \mathcal{L}_{Y_i}(d(\iota_{Y_i}\omega)+\iota_{Y_i} d\omega)\\
&= \mathcal{L}_{Y_i}(d(-\sum_{j=1}^n \lambda_i^j d{f_j})) \\
&= -\mathcal{L}_{Y_i}\left(\sum_{j=1}^n d\lambda_i^j \wedge d{f_j}\right) = 0
\end{align}

\noindent where in the last equality we used the fact that $\lambda_i^j$ are constant on the level sets of $F$. Recall from \cite{Laurent-Gengoux2010} (see Claim 2 in page 1856) that if $Y$ is any complete periodic vector field, and $P$ is a bivector field for which $\mathcal{L}_Y\mathcal{L}_Y P = 0$, then $\mathcal{L}_Y  P= 0$. This proves that $\mathcal{L}_{Y_i}\omega = 0$, so the vector fields $Y_i$ are Poisson vector fields.

To show that each $\iota_{Y_i}\omega$ is has a $^bC^{\infty}$ primitive, it suffices to show that the smooth part of $[\iota_{Y_i}\omega]$, i.e. the first summand in its image under the Mazzeo-Melrose isomorphism (see Theorem \ref{thm:mazzeomelrose}) ${^b}H^1(\mathbb{T}^n \times B^n) \cong H^1(\mathbb{T}^n \times B^n) \oplus H^0(\mathbb{T}^n \times B^n)$, is zero. But this follows from the fact that the value of the smooth part $H^1(\mathbb{T}^n \times B^n)$ of $\iota_{Y_i}\omega$ is determined by integrating it along loops which are tangent to the fibers. Since the kernel of $\iota_{Y_i}\omega$ contains the tangent space to the torus fibers, these integrals are zero.

Hence the $Y_i$ are $b$-Hamiltonian and we denote their Hamiltonian functions by $\sigma_i$. Because $\lambda_i^n$ vanish along $B^{n-1}\times\{0\}$ for $i < n$, the forms $\iota_{Y_i}\omega = -\sum_{j=1}^n \lambda_i^j d{f_j}$ are smooth for $i < n$, so $\sigma_i$ are smooth for $i < n$. Because $\lambda_n^n$ equals $c$ along $B^{n-1}\times\{0\}$, $\sigma_n$ has the form $c\log|t|$ for some defining function $t$.

\paragraph{\emph{Step 3 -- apply the Darboux-Carath\'{e}odory theorem.}}

The construction above gives us candidates $\sigma_1,\ldots, \sigma_n = c\log|t|$ for the the ``action coordinates.''
By the Darboux-Carath\'{e}odory theorem we can \emph{locally} find a coordinate system
\[
(\sigma_1, \dots, \sigma_{n-1}, t, q_1,\ldots, q_n)
\]
such that
\begin{equation}\label{eq:aab}
\omega = \sum_{i=1}^{n-1} d\sigma_i \wedge dq_i + \frac{c}{t} d{t}\wedge dq_n.
\end{equation}
Since the vector fields $X_{\sigma_i}= \frbd{q_i} (i=1,\ldots,n)$ are the fundamental vector fields of the $\T^{n}$-action, in the local chart introduced above the flow of the vector fields gives a linear action on the $q_i$ coordinates.

Therefore, if the functions $\sigma_1,\ldots, \sigma_{n-1}, t$ were initially defined on an open set $U\subset M$, we can extend them to the whole set $U':=p^{-1}(p(U))$ (i.e. the union of all tori that intersect non-trivially with $U$). We denote the extensions of these functions by the same symbols.

The vector fields $\frbd{q_i}(i=1,\ldots,n)$ have period $1$, so we can view $q_i (i=1,\ldots,n)$ as $S ^1$ valued coordinates. We denote them by the ``angle'' variable $\theta_i$ for this reason.

It remains to check that the extended functions $(\sigma_1, \dots, \sigma_{n-1}, t, \theta_1, \dots ,\theta_n)$ define a coordinate system on $U'$ and that $\omega$ still has the form

\begin{equation}\label{eq:aab2}
\omega = \sum_{i=1}^{n-1} d\sigma_i \wedge d\theta_i + \frac{c}{t} d{t}\wedge d\theta_n.
\end{equation}

It is clear that $\{\sigma_i, \theta_j\}=\delta_{ij}$ on $U'$. To show that $\{\theta_i, \theta_j\}=0$, we note that this relation holds on $U$ and flowing with the vector fields $X_{\sigma_k}$ we see that it holds on the whole set $U'$:
$$ X_{\sigma_k}\big(\{\theta_i,\theta_j\} \big) = \{\{\theta_i, \theta_j\},  \sigma_k\} = \{ \theta_i, \delta_{ij}\} - \{\theta_j,\delta_{ik}\} = 0. $$
This verifies that $\omega$ has the form \eqref{eq:aab2} above and in particular, we conclude that the derivatives of the functions $ \sigma_1, \theta_1,\ldots,  \sigma_{n-1}, \theta_{n-1}, t, \theta_n$ are independent on $U$, hence these functions define a coordinate system. This completes the proof of the action-angle coordinate theorem.

\section{KAM Theorem for $b$-integrable systems}\label{sec:kamb}

We will prove a stability result of KAM type for Hamiltonian systems on $b$-Poisson manifolds, allowing Hamiltonian functions which are $b$-functions. For a given Hamiltonian function $H\in \,^b C^\infty(M)$ the Hamiltonian equations are completely analogous to the symplectic case: Any function $g\in C^\infty(M)$ evolves according to the equation $\dot g = \{g,H\} = X_H(g)$.

Assume that we are given a $b$-integrable system on $M$. From the action-angle coordinate theorem for $b$-symplectic manifolds, Theorem~\ref{thm:aa}, it follows that we can semilocally (around a Liouville torus) replace the given $b$-integrable system by the functions $y_1,\ldots,y_n$ on $\T^n \times B^n$, where the $y_i$ are the projections to the $i$-th component of $B^n$ and the $b$-symplectic structure is
$$\frac{c}{y_1} d{\varphi_1}\wedge d{y_1} + \sum_{i=2}^{n-1} d{\varphi_i}\wedge d{y_i}.$$
We want to study the stability of quasi-periodic motion inside the exceptional hypersurface $Z = \{ y_1 = 0 \}$ of $(\T^n \times B^n, \omega)$. As in the symplectic case, we consider a Hamiltonian ``compatible'' with the integrable system in the sense that it only depends on the action coordinates $y$,
\begin{equation}\label{b-hamiltonian}
H(\varphi, y ) = k \log|y_1| + h(y).
\end{equation}
The resulting equations of motion are
\begin{align}\label{b-ham_sys_eom}
   \begin{split}
      &\dot { \varphi_1} = \frac{k}{c} + \frac{y_1}{c} \frbdt{h}{y_1}(y) \\
      &\dot { \varphi_i} =  \frbdt{h}{y_i}(y) \ \ \qquad i=2,\ldots,n \\
      &\dot {y_1} =  0  \\
      &\dot {y_i} =  0 \qquad i=2,\ldots,n .
   \end{split}
\end{align}

In terms of notation, if $x\in \R^n$, we will write $\tilde x$ for the $\R^{n-1}$ vector obtained by omitting the first component $\tilde x := (x_2, \ldots, x_n)$. We see that on $Z$ the angle coordinates $(\varphi_1, \ldots, \varphi_n)$ evolve with frequency
$$\big (\frac{k}{c}, \frbdt{h}{y_2}(y), \ldots, \frbdt{h}{y_n}(y)) =: ( \frac{k}{c} , \tilde \omega \big)$$
on an $n$-torus $\T^n \times \{\const\} \subset Z$.

We want to study the effect of adding a perturbation $\epsilon P\in\, ^b \! C^\infty(M)$. Writing
\begin{equation}\label{perturbb}
 P (\varphi,y) = k' \log|y_1| + \epsilon f(\varphi, y )
\end{equation}
 the equations of motion of the perturbed system are:
\begin{align}\label{b-ham_sys_perturbed_eom}
    \begin{split}
       &\dot { \varphi_1} = \frac{k + \epsilon k'}{c} + \frac{y_1}{c} \frbd{y_1} (h(y)+ \epsilon f(\varphi,y)) \\
       &\dot { \varphi_i} =  \frbd{y_i}(g(y)+ f(\varphi,y)) = \tilde \omega_i(y) + \epsilon \frbdt{f}{y_i}(\varphi,y) , \qquad i=2,\ldots,n \\
       &\dot {y_1} =  -\epsilon \frac{y_1}{c} \frbdt{f}{\varphi_1}(\varphi,y) \\
       &\dot {y_i} =  - \epsilon \frbdt{f}{\varphi_i}(\varphi,y).
    \end{split}
\end{align}

Notice that the hypersurface $Z = \{y_1 = 0\}$ is preserved by the perturbation. In the following discussion of stability of the orbits we restrict ourselves to the case where the motion starts inside $Z$ (and necessarily remains there);  the other case is the classical KAM theorem for symplectic manifolds.

We want to consider the case where the function $f$ in the expression \eqref{perturbb} for the perturbation $P$ has the form
\begin{equation}\label{assumptionf}
 f(\varphi, y) = f_1(\tilde \varphi, y ) + y_1 f_2(\varphi, y) + f_3(\varphi_1,y_1),
\end{equation}
where $f_1$ is an analytic function and $f_2, f_3$ are smooth functions. In particular, this case occurs if $f$ does not depend on $\varphi_1$.

\begin{theorem}[KAM Theorem for $b$-Poisson manifolds]

Let $\T^n \times B_r^n$ be endowed with the standard $b$-symplectic structure. Consider a $b$-function $H = k \log|y_1| + h(y)$ on this manifold, where $h$ is analytic. Let $y_0$ be a point in $B_r^n$ with first component equal to zero, so that the corresponding level set $\T^n \times \{y_0\}$ lies inside the exceptional hypersurface $Z$.

Assume that the frequency map $\tilde \omega( y):= \frbdt{h}{\tilde y}(y)$ has a Diophantine value $\tilde \omega := \tilde \omega(y_0)$ at $y_0 \in B^n$ and that it is non-degenerate at $y_0$ in the sense that the Jacobian $ \frbdt{\tilde \omega}{\tilde y} (y_0) $ is regular.

Then the torus $\T^n \times \{ {y_0}\}$  persists under sufficiently small perturbations of $H$ which have the form mentioned above, i.e. they are given by $\epsilon P$, where $\epsilon \in \R$ and $P \in ^b \! C^\infty(\T^n \times B_r^n)$ has the form
\begin{align*}
 P(\varphi, y) &= \log |y_1| + f(\varphi,y) \\
 f(\varphi, y) &= f_1(\tilde \varphi, y ) + y_1 f_2(\varphi, y) + f_3(\varphi_1,y_1).
\end{align*}
More precisely, if $|\epsilon|$ is sufficiently small, then the  perturbed system
  $$ H_\epsilon =H + \epsilon P$$
admits an invariant torus $\mathcal{T}$.

Moreover, there exists a diffeomorphism $\T^n \to \mathcal{T}$ close\footnote{By saying that the diffeomorphism is ``$\epsilon$-close to the identity'' we mean that, for given $H, P$ and $r$, there is a constant $C$ such that $\norm{\psi - \id} < C \epsilon.$} to the identity taking the flow $\gamma^t$  of the perturbed system on $\mathcal{T}$ to the linear flow  on $\T^n$ with frequency vector  $ (\frac{k+\epsilon k'}{c}, \tilde \omega)$.
\end{theorem}

\begin{proof}
First assume that $y_0 = 0$. We will prove the general case later on. We consider the restrictions of $h$ and $f_1$ to $Z$ as functions on $B_r^{n-1}$ resp. $\T^{n-1} \times B_r^{n-1}$:
 $$\overline h (\tilde y) := h(0, \tilde y), \qquad \overline f_1(\tilde \varphi, \tilde y) := f_1(\tilde \varphi, 0, \tilde y).$$
 By the Kolmogorov theorem for symplectic manifolds, Theorem \ref{thm:kolmogorov}, there exists a constant $\epsilon_0>0$ such that for $0<\epsilon < \epsilon_0$ there is a symplectomorphism $\overline \psi: \T^{n-1} \times B^{n-1}_{r_\ast} \to \T^{n-1} \times B^{n-1}_r $ such that
 $$ (\overline h + \epsilon \overline f_{1})\circ \overline \psi = h_\ast$$
 is a function in Kolmogorov normal form on $\T^{n-1} \times B^{n-1}_{r_\ast}$ with frequency vector $\tilde \omega$. Denoting
 $\overline \psi(\tilde \varphi, \tilde y )=: (\tilde \varphi', \tilde y ')$ we define a Poisson diffeomorphism
 $$ \psi: \T^{n} \times B^{n}_{r_\ast} \to \T^{n} \times B^{n}_r  , \quad  \psi(\varphi, y):=(\varphi_1, \tilde \varphi', y_1, \tilde y').$$
Then for $(\varphi, y ) \in Z \subset \T^n \times B_{r_\ast}^n$:
 \begin{align*}
   (H+&\epsilon P) \circ  \psi (\varphi, y ) = (H+\epsilon P) (\varphi_1, \tilde \varphi', y_1, \tilde y') = \\
   = &(k+\epsilon k') \log |y_1| + \epsilon y_1 f_2(\varphi_1, \tilde \varphi', y_1, \tilde y') + \epsilon f_3(\varphi_1,y_1) + \\
   & + h(y_1, \tilde y') +\epsilon f_1( \tilde \varphi', y_1, \tilde y') .
 \end{align*}
By the argument above, on $Z$ (i.e. $y_1 = 0$) the two last terms can be written as
  \begin{align*}
h(y_1, \tilde y') +\epsilon f_1( \tilde \varphi', y_1, \tilde y') = \overline  h(\tilde y') + \epsilon \overline f(\tilde \varphi',\tilde y')   = h_\ast(\tilde \varphi, \tilde y) .
 \end{align*}
Keeping in mind that $h_\ast$ is of Kolmogorov normal form with frequency vector $\tilde \omega$,
$$h_\ast  (\tilde \varphi, \tilde y) = E_\ast + \tilde \omega \times \tilde y + Q_\ast(\tilde \varphi, \tilde y), $$
where $E_\ast \in \R$, $Q_\ast = \SO(|y|²)$, we see by looking at the equations of motion \eqref{b-ham_sys_perturbed_eom} that the trajectories of  $H_\ast := (H+\epsilon P) \circ \psi$ on $Z$ are precisely given by quasi-periodic motion with frequency $ (\frac{k+\epsilon k'}{c}, \tilde \omega)$ on the tori $\T^n \times \{m\}$ for $m \in B_{r_\ast}$.

The flow of the Hamiltonian vector field associated to $H+\epsilon P$ is conjugated under $\psi$ to the flow of the Hamiltonian vector field associated to $H_\ast$:
$$ \gamma^t_{X_{H+\epsilon P}  } = \psi \circ \gamma^t_{X_{H_\ast}} \circ \psi^{-1}.$$
Since the flow $\gamma^t_{X_{H_\ast}}$ leaves the torus $\T^n\times \{0\}$ invariant, the flow of the perturbed system $\gamma^t_{X_{H+\epsilon P}  }$ leaves $\ST:= \psi(\T^n \times \{0\}) \cong \T^n$ invariant.

In conclusion, the motion induced by $H+ \epsilon P$ on $\ST$ is conjugated via $\psi$ to the quasi-periodic motion on $\T^n\times \{0\}$ with frequency $(\frac{k+\epsilon k'}{c}, \tilde \omega)$.

Since the diffeomorphism $\overline \psi$ obtained from the Kolmogorov theorem for symplectic manifolds is $\epsilon$-close to the identity, the transformation $\psi$ we construct is also $\epsilon$-close to the identity. \\

Now consider the case where $y_0 \neq 0$. Let $\tau$ be the translation which takes $y_0$ to $0$. Note that $\tau$ only changes the last $n-1$ components since we assume that the first component of $y_0$ is already 0, so in particular $\tau$ is a Poisson diffeomorphism
$$\T^n\times B_{r'}(y_0) \to  \T^n\times B_{r'}(0), $$
where $B_{r'}(y_0)$ is a ball around $y_0$ of some radius $r'>0$ contained in the original ball $B_r^n$ and we endow the sets with the $b$-symplectic structure inherited from $\T^n\times B_r^n$.
Now we apply the argument above (case $y_0 = 0$) to the Hamiltonian $H\circ \tau$ and the perturbation $P\circ \tau$. Denote the diffeomorphism obtained there by $\psi_0$,
$$\psi_0 : \T^n\times B^n_{r_\ast'}\to  \T^n\times B^n_{r'}.$$
Setting $\psi := \tau \circ \psi_0 \circ \tau^{-1}$ we obtain a Poisson diffeomorphism $\T^n\times B_{r_\ast'}(y_0) \to \T^n\times B^n_{r'}(y_0)$ which is $\epsilon$-close to the identity and conjugates the motion on $\ST := \psi(\T^n \times \{y_0\})$ to quasi-periodic motion on $\T^n \times \{y_0\}$ with frequency $(\frac{k+\epsilon k'}{c}, \tilde \omega)$.
\end{proof}

\begin{remark}
 Note that the first component of the frequency vector changes; only the last $n-1$ components of $\omega$ are preserved. Moreover, since we only assume the Diophantine condition for the last $n-1$ components, the orbit through $p_0$ might not fill the whole torus $\T^n \times \{y_0\}$ densely. However, even in these cases the torus $\T^n \times \{y_0\}$ is invariant.
\end{remark}

\begin{remark}
 A special case is where the functions  $H$ and $P$ are smooth, i.e. the $\log$-component is zero:
$$H = h \in C^\infty(B_r^n), \qquad P = f \in C^\infty(\T^n \times B_r^n).$$
 In this case we do not have to make the assumption that $f$ has the form given in Equation \eqref{assumptionf} to obtain stability of the orbits inside $Z$. From the equations of motion it is clear that the trajectory starting at a point inside a symplectic leaf $\SL \subset Z$ will stay inside the leaf. This is true also after adding the perturbation. Hence the stability of the orbit follows directly from the symplectic KAM theorem: If $H$ is in $b$-Kolmogorov normal form with vanishing $\log$ component (i.e. a $C^\infty$ function) and with Diophantine frequency vector  $\tilde \omega$ and if $P$ is any $C^\infty$ perturbation, then there is a symplectomorphism on a neighborhood of the orbit {\it inside $\SL$} which is close to the identity and takes the perturbed orbit to a nearby $n-1$ torus $\{\varphi_1\} \times \T^{n-1} \times \{y\}$. The perturbed motion is conjugated to linear motion in the $\tilde \varphi := (\varphi_2,\ldots,\varphi_n)$ coordinates with frequency $\tilde \omega$. Note that we only
transform
inside the leaf $\SL$ here -- for showing stability
this is sufficient.
\end{remark}

\begin{remark} In view of example \ref{manifoldswithboundary}, this KAM theorem can be useful to study perturbations of integrable systems on manifolds with boundary.
\end{remark}

\end{document}